\documentclass{article}%
\usepackage{amsfonts}
\usepackage{amsmath}
\usepackage{amssymb}
\usepackage{graphicx}%
\providecommand{\U}[1]{\protect\rule{.1in}{.1in}}
\newtheorem{theorem}{Theorem}

\newtheorem{lemma}[theorem]{Lemma}

\newtheorem{proposition}[theorem]{Proposition}
\newtheorem{remark}[theorem]{Remark}

\newenvironment{proof}[1][Proof]{\noindent\textbf{#1.} }{\ \rule{0.5em}{0.5em}}
\begin{document}

\title{EULER SUMS OF HYPERHARMONIC NUMBERS}
\author{\textbf{Ayhan Dil}\\Department of Mathematics,\\Akdeniz University, 07058 Antalya Turkey\\e-mail: adil@akdeniz.edu.tr \\\textbf{Khristo N. Boyadzhiev}\\Department of Mathematics and Statistics,\\Ohio Northern University Ada, \ Ohio 45810, USA,\\e-mail: k-boyadzhiev@onu.edu}
\date{}
\maketitle

\begin{abstract}
The hyperharmonic numbers $h_{n}^{\left(  r\right)  }$ are defined by means of
the classical harmonic numbers. We show that the Euler-type sums with
hyperharmonic numbers:%
\[
\sigma\left(  r,m\right)  =%
{\displaystyle\sum\limits_{n=1}^{\infty}}
\frac{h_{n}^{\left(  r\right)  }}{n^{m}}%
\]
can be expressed in terms of series of Hurwitz zeta function values. This is a
generalization of a result of Mez\H{o} and Dil. We also provide an explicit
evaluation of $\sigma\left(  r,m\right)  $ in a closed form in terms of zeta
values and Stirling numbers of the first kind. Furthermore, we evaluate
several other series involving hyperharmonic numbers.\newline\textbf{2010
Mathematics Subject Classification:} 11B73, 11M99\newline\textbf{Keywords and
Phrases:} Riemann zeta function, Hurwitz zeta function, Euler sums, harmonic
and hyperharmonic numbers, Stirling numbers, Beta function.

\end{abstract}

\section{Introduction}

In this paper we are interested in Euler-type sums with hyperharmonic numbers
$\sigma\left(  r,m\right)  $ (see definitions $\left(  \ref{i0}\right)  $\ and
$\left(  \ref{i1}\right)  $\ below). Such series could be of interest in
analytic number theory. We will show that these sums are related to the values
of the Riemann zeta function. In \cite{MD}\ the authors considered the case
$r=1.$\ Here we extend this result to $r>1$.

In the second section we express $\sigma\left(  r,m\right)  $ as a special
series involving zeta values. In the third section we evaluate $\sigma\left(
r,m\right)  $ as a finite sum including Stirling numbers of the first kind,
zeta values, and values of the digamma (psi) function.

In the last fourth section we use certain integral representaions to evaluate
several series with hyperharmonic numbers. For example,%
\[
\sum_{n=1}^{\infty}\frac{h_{n}^{\left(  r\right)  }}{n\left(  n+1\right)
\ldots\left(  n+r\right)  }=\frac{\pi^{2}}{6r!}%
\]
and%
\[
\sum_{n=1}^{\infty}h_{n}^{\left(  r\right)  }B\left(  r+1,n+1\right)  =1.
\]

\subsection{Definitions and notation}

The $n$-th harmonic number is defined by the $n$-th partial sum of the
harmonic series%
\begin{equation}
H_{n}:=\sum_{k=1}^{n}\frac{1}{k}\text{ \ \ }\left(  n\in\mathbb{N}:=\left\{
1,2,3,\ldots\right\}  \right)  , \label{i0}%
\end{equation}

where the empty sum $H_{0}$ is conventionally understood to be zero.

Starting with $h_{n}^{(0)}=\frac{1}{n}$ $\left(  n\in\mathbb{N}\right)  ,$ the
$nth$ hyperharmonic number $h_{n}^{(r)}$ of order $r$ is defined by (see
\cite{CG}, see also \cite{MD}):%

\begin{equation}
h_{n}^{(r)}:=\sum_{k=1}^{n}h_{k}^{(r-1)}\text{ \ \ }\left(  r\in
\mathbb{N}\right)  . \label{i1}%
\end{equation}
It is easy to see that $h_{n}^{(1)}:=H_{n}$ $\left(  n\in\mathbb{N}\right)  $.

These numbers can be expressed in terms of binomial coefficients and ordinary
harmonic numbers (see \cite{CG, MD}):%
\begin{equation}
h_{n}^{(r)}=\binom{n+r-1}{r-1}(H_{n+r-1}-H_{r-1}). \label{i2}%
\end{equation}

The well-known generating functions of the harmonic and hyperharmonic numbers
are given as%
\begin{equation}
\sum_{n=1}^{\infty}H_{n}x^{n}=-\frac{\ln\left(  1-x\right)  }{1-x} \label{i3}%
\end{equation}
and%
\begin{equation}
\sum_{n=1}^{\infty}h_{n}^{\left(  r\right)  }x^{n}=-\frac{\ln\left(
1-x\right)  }{\left(  1-x\right)  ^{r}}. \label{i4}%
\end{equation}

Euler discovered the following formula (see, e.g., \cite{FS, RS})):%

\begin{equation}
2\text{ }\zeta_{H}\left(  m\right)  =\left(  m+2\right)  \zeta\left(
m+1\right)  -\sum_{n=1}^{m-2}\zeta\left(  m-n\right)  \zeta\left(  n+1\right)
\text{\ \ \ }\left(  m\in\mathbb{N}\backslash\left\{  1\right\}  \right)  ,
\label{1}%
\end{equation}
where $\zeta_{H}\left(  m\right)  =\sum_{n=1}^{\infty}\frac{H_{n}}{n^{m}}$ and
$\zeta\left(  s\right)  $ is the Riemann zeta function, and, throughout this
paper the empty sum understood to be nil .

For certain pairs of positive integers $r$ and $m,$ several authors have
evaluated the Euler sums%
\[
S\left(  r,m\right)  =\sum_{n=1}^{\infty}\frac{H_{n}^{\left(  r\right)  }%
}{n^{m}},
\]
where $H_{n}^{\left(  r\right)  }$ denotes the generalized harmonic numbers of
order $r$ defined by
\[
H_{n}^{\left(  r\right)  }=1+\frac{1}{2^{r}}+\cdots+\frac{1}{n^{r}}=\sum
_{k=1}^{n}\frac{1}{k^{r}}%
\]
\ (see \cite{BBG, FS} and for an elementary procedure \cite{BC}). In the
earlier works have been done with $S\left(  r,m\right)  $, in this paper, we
want to give a closed form of the following sum:%
\[
\sigma\left(  r,m\right)  =%
{\displaystyle\sum\limits_{n=1}^{\infty}}
\frac{h_{n}^{\left(  r\right)  }}{n^{m}}.
\]

\bigskip Using the relation $\left(  \ref{i2}\right)  ,$ Mez\H{o} and Dil
(\cite{MD}, Corollary 3) found that the series $\sigma\left(  r,m\right)  $
converges for $m>r$, i.e.,%
\[
\sigma\left(  r,m\right)  =\sum_{n=1}^{\infty}\frac{h_{n}^{\left(  r\right)
}}{n^{m}}<\infty\text{ \ \ }\left(  m>r\right)  .
\]
A rearrangement transforms $\sigma\left(  r,m\right)  $ into the following sum
(\cite{MD}, Theorem 4) :%
\begin{equation}
\sigma\left(  r,m\right)  =\sum_{n=1}^{\infty}h_{n}^{\left(  r-1\right)
}\zeta\left(  m,n\right)  \text{ \ \ }\left(  r\in\mathbb{N}\text{; }m\geq
r+1\right)  , \label{2}%
\end{equation}
where $\zeta\left(  s,a\right)  $ is the Hurwitz (or generalized) zeta
function defined by%

\[
\zeta\left(  s,a\right)  =\sum_{n=0}^{\infty}\frac{1}{\left(  n+a\right)
^{s}}\text{ \ \ }\left(  Re\left(  s\right)  >1\text{; }a\notin\mathbb{Z}%
_{0}^{-}\right)  ,
\]
where $\mathbb{Z}_{0}^{-}:=\left\{  0,-1,-2,...\right\}  .$

Using $\left(  \ref{1}\right)  $ and $\left(  \ref{2}\right)  ,$\ Mez\H{o} and
Dil (\cite{MD}, p. 364) obtained the following identity:%
\begin{equation}
2\sum_{k=1}^{\infty}\frac{\zeta\left(  m,k\right)  }{k}=\left(  m+2\right)
\zeta\left(  m+1\right)  -\sum_{n=1}^{m-2}\zeta\left(  m-n\right)
\zeta\left(  n+1\right)  \text{ \ \ }\left(  m\in\mathbb{N}\backslash\left\{
1\right\}  \right)  . \label{3}%
\end{equation}

\section{\textbf{Euler Sums of Hyperharmonic Numbers}}

The following theorem provides a general version of the Equation $\left(
\ref{2}\right)  $.

\begin{theorem}
\label{D}For $0\leq k<r<m$ ($k,r,m\in\mathbb{N}_{0}:=\mathbb{N\cup}\left\{
0\right\}  $), we have%
\begin{equation}
\sigma\left(  r,m\right)  =\sum_{n=1}^{\infty}\frac{h_{n}^{\left(  r\right)
}}{n^{m}}=\sum_{n=1}^{\infty}h_{n}^{\left(  r-k-1\right)  }\sum_{n\leq
i_{1}\leq i_{2}\leq\cdots\leq i_{k}<\infty}\zeta\left(  m,i_{k}\right)  ,
\label{9}%
\end{equation}
where%
\begin{equation}
\sum_{n\leq i_{1}\leq i_{2}\leq\cdots\leq i_{k}<\infty}\zeta\left(
m,i_{k}\right)  =\sum_{i_{1}=n}^{\infty}\sum_{i_{2}=i_{1}}^{\infty}\ldots
\sum_{i_{k}=i_{k-1}}^{\infty}\zeta\left(  m,i_{k}\right)  . \label{10}%
\end{equation}

\end{theorem}

\begin{proof}
From $\left(  \ref{2}\right)  $\ we have%
\begin{equation}
\sigma\left(  r,m\right)  =\sum_{n=1}^{\infty}\frac{h_{n}^{\left(  r\right)
}}{n^{m}}=\sum_{n=1}^{\infty}\sum_{j=1}^{n}h_{j}^{\left(  r-2\right)  }%
\zeta\left(  m,n\right)  . \label{11}%
\end{equation}
A rearrangement of $\left(  \ref{11}\right)  $ gives%
\begin{equation}
\sigma\left(  r,m\right)  =\sum_{n=1}^{\infty}h_{n}^{\left(  r-2\right)  }%
\sum_{i_{1}=n}^{\infty}\zeta\left(  m,i_{1}\right)  . \label{12}%
\end{equation}
Using a similar argument, after $k$-steps we obtain the desired result.
\end{proof}

\begin{remark}
Let us consider two special cases of $\left(  \ref{9}\right)  .$ To express
the Euler sums of hyperharmonic numbers $\sigma\left(  r,m\right)  $ in terms
of the multiple sums of Hurwitz zeta function we set $k=r-1$ to have%
\begin{equation}
\sigma\left(  r,m\right)  =\sum_{n=1}^{\infty}\frac{\sum_{n\leq i_{1}\leq
i_{2}\leq\cdots\leq i_{r-1}<\infty}\zeta\left(  m,i_{r-1}\right)  }{n}.
\label{c1}%
\end{equation}
Also the case $k=r-2$ gives%
\begin{equation}
\sigma\left(  r,m\right)  =\sum_{n=1}^{\infty}H_{n}\sum_{n\leq i_{1}\leq
i_{2}\leq\cdots\leq i_{r-2}<\infty}\zeta\left(  m,i_{r-2}\right)  , \label{c2}%
\end{equation}
which is a representation of $\sigma\left(  r,m\right)  $ in terms of harmonic
numbers and the multiple sums of Hurwitz zeta function.
\end{remark}

\section{A closed form of $\sigma\left(  r,m\right)  $}

We shall present a closed form evaluation of the following sum:%
\[
\sigma\left(  r,m\right)  =%
{\displaystyle\sum\limits_{n=1}^{\infty}}
\frac{h_{n}^{\left(  r\right)  }}{n^{m}}\text{ , where }m\geq r+1.
\]
In the next theorem we use a known closed form evaluation of the following
series recalled in Lemma $3$ below:%
\[
\mu\left(  m,r\right)  =%
{\displaystyle\sum\limits_{n=1}^{\infty}}
\frac{1}{n^{m}\left(  n+r\right)  }%
\]
and the unsigned Stirling numbers of the first kind $%
\genfrac{[}{]}{0pt}{}{n}{k}%
.$

\begin{lemma}
( \cite{BE}) For every positive integer $m$ and every $r>0$
\begin{equation}
\mu\left(  m,r\right)  =\sum_{k=1}^{m-1}\frac{\left(  -1\right)  ^{k-1}%
\zeta\left(  m-k+1\right)  }{r^{k}}+\frac{\left(  -1\right)  ^{m-1}}{r^{m}%
}\left(  \Psi\left(  r+1\right)  +\gamma\right)  , \label{L}%
\end{equation}
where $\Psi\left(  s\right)  $ is the Psi (or digamma) function defined by
$\Psi\left(  s\right)  =\frac{\Gamma^{^{\prime}}\left(  s\right)  }%
{\Gamma\left(  s\right)  }$\ and $\gamma=-\Psi\left(  1\right)  $ is the
Euler-Mascheroni constant.
\end{lemma}

\begin{theorem}
\label{B}For $r,m\in\mathbb{N}$ with $r<m,$ we have
\begin{align}
\sigma\left(  r,m\right)   &  =\frac{1}{\left(  r-1\right)  !}%
{\displaystyle\sum\limits_{k=1}^{r}}
\genfrac{[}{]}{0pt}{}{r}{k}%
\nonumber\\
&  .\left\{  \zeta_{H}\left(  m-k+1\right)  -H_{r-1}\zeta\left(  m-k+1\right)
+%
{\displaystyle\sum\limits_{j=1}^{r-1}}
\mu\left(  m-k+1,j\right)  \right\}  . \label{cc3}%
\end{align}

\end{theorem}

\begin{proof}
Using the well-known expansion:%
\[
x\left(  x+1\right)  \ldots\left(  x+n-1\right)  =%
{\displaystyle\sum\limits_{k=0}^{n}}
\genfrac{[}{]}{0pt}{}{n}{k}%
x^{k},
\]
we get%
\[
\binom{n+r}{r}=\frac{1}{r!}%
{\displaystyle\sum\limits_{k=1}^{r+1}}
\genfrac{[}{]}{0pt}{}{r+1}{k}%
n^{k-1}.
\]
Hence,%
\begin{align*}
h_{n}^{\left(  r+1\right)  }  &  =\frac{1}{r!}%
{\displaystyle\sum\limits_{k=1}^{r+1}}
\genfrac{[}{]}{0pt}{}{r+1}{k}%
n^{k-1}\left(  H_{n+r}-H_{r}\right) \\
&  =\frac{1}{r!}%
{\displaystyle\sum\limits_{k=1}^{r+1}}
\genfrac{[}{]}{0pt}{}{r+1}{k}%
n^{k-1}\left(  H_{n}+\frac{1}{n+1}+\cdots+\frac{1}{n+r}-H_{r}\right)  .
\end{align*}
Now we can write%
\[
\frac{h_{n}^{\left(  r\right)  }}{n^{m}}=\frac{1}{\left(  r-1\right)  !}%
{\displaystyle\sum\limits_{k=1}^{r}}
\genfrac{[}{]}{0pt}{}{r}{k}%
\left(  \frac{H_{n}}{n^{m-k+1}}-\frac{H_{r-1}}{n^{m-k+1}}+%
{\displaystyle\sum\limits_{j=1}^{r-1}}
\frac{1}{n^{m-k+1}\left(  n+j\right)  }\right)  .
\]
Therefore,%
\begin{align*}
\sum\limits_{n=1}^{\infty}\frac{h_{n}^{\left(  r\right)  }}{n^{m}}  &
=\frac{1}{\left(  r-1\right)  !}%
{\displaystyle\sum\limits_{k=1}^{r}}
\genfrac{[}{]}{0pt}{}{r}{k}%
\left\{
{\displaystyle\sum\limits_{n=1}^{\infty}}
\frac{H_{n}}{n^{m-k+1}}\right. \\
&  \left.  -H_{r-1}%
{\displaystyle\sum\limits_{n=1}^{\infty}}
\frac{1}{n^{m-k+1}}+%
{\displaystyle\sum\limits_{j=1}^{r-1}}
{\displaystyle\sum\limits_{n=1}^{\infty}}
\frac{1}{n^{m-k+1}\left(  n+j\right)  }\right\}  .
\end{align*}
In view of $\left(  \ref{L}\right)  ,$ we obtain $\left(  \ref{cc3}\right)  .$
\end{proof}

Next we consider two particular cases $r=2$ and $3$ of $\sigma\left(
r,m\right)  .$

\textbf{Case }$r=2.$

By considering the case $r=2$ in $\left(  \ref{cc3}\right)  $\ we get%
\begin{equation}
\sigma\left(  2,m\right)  =\zeta_{H}\left(  m-1\right)  +\zeta_{H}\left(
m\right)  -\zeta\left(  m-1\right)  . \label{4}%
\end{equation}
With the aid of Theorem \ref{D}, equation $\left(  \ref{4}\right)  $ can be
written as%
\begin{equation}
\sigma\left(  2,m\right)  =\sum_{n=1}^{\infty}H_{n}\zeta\left(  m,n\right)
=\zeta_{H}\left(  m-1\right)  +\zeta_{H}\left(  m\right)  -\zeta\left(
m-1\right)  . \label{4+}%
\end{equation}

Setting $m=3$ in $\left(  \ref{4+}\right)  $ we obtain%
\begin{equation}
\sigma\left(  2,3\right)  =\sum_{n=1}^{\infty}H_{n}\zeta\left(  3,n\right)
=2\zeta\left(  3\right)  +\frac{5}{4}\zeta\left(  4\right)  -\zeta\left(
2\right)  \label{5}%
\end{equation}
and for $m=4$ we have%
\begin{equation}
\sigma\left(  2,4\right)  =\sum_{n=1}^{\infty}H_{n}\zeta\left(  4,n\right)
=\frac{5}{4}\zeta\left(  4\right)  +3\zeta\left(  5\right)  -\zeta\left(
2\right)  \zeta\left(  3\right)  -\zeta\left(  3\right)  \label{6}%
\end{equation}
and so on.

\textbf{Case }$r=3.$

By $\left(  \ref{9}\right)  $ and $\left(  \ref{cc3}\right)  $ we have%

\begin{align*}
\sigma\left(  3,m\right)   &  =\sum_{n=1}^{\infty}h_{n}^{\left(  2\right)
}\zeta\left(  m,n\right) \\
&  =\frac{1}{2}\zeta_{H}\left(  m-2\right)  +\frac{3}{2}\zeta_{H}\left(
m-1\right)  +\zeta_{H}\left(  m\right)  -\frac{5}{4}\zeta\left(  m-1\right)
-\frac{3}{4}\zeta\left(  m-2\right)
\end{align*}
Setting $m=4$ in the above equation\ we obtain%
\begin{equation}
\sigma\left(  3,4\right)  =\sum_{n=1}^{\infty}h_{n}^{\left(  2\right)  }%
\zeta\left(  4,n\right)  =\frac{15}{6}\zeta\left(  4\right)  +\zeta_{H}\left(
4\right)  -\frac{1}{4}\zeta\left(  3\right)  -\frac{3}{4}\zeta\left(
2\right)  . \label{8}%
\end{equation}

\section{Some series with hyperharmonic numbers}

In this section we evaluate some specific series involving hyperharmonic numbers.

\begin{proposition}
\label{P5}%
\begin{equation}
\sum_{n=1}^{\infty}\frac{h_{n}^{\left(  r\right)  }}{\left(  n+1\right)
\left(  n+2\right)  \ldots\left(  n+r+1\right)  }=\frac{1}{r!}\text{ }\left(
r\in\mathbb{N}_{0}\right)  . \label{13}%
\end{equation}

\end{proposition}

\begin{proof}
Using the formula (see \cite{GR})%
\begin{equation}
\frac{1}{r!}\int_{0}^{1}t^{n}\left(  1-t\right)  ^{r}dt=\frac{1}{\left(
n+1\right)  \left(  n+2\right)  \ldots\left(  n+r+1\right)  }, \label{14}%
\end{equation}
we can write%
\begin{equation}
\frac{1}{r!}\int_{0}^{1}h_{n}^{\left(  r\right)  }t^{n}\left(  1-t\right)
^{r}dt=\frac{h_{n}^{\left(  r\right)  }}{\left(  n+1\right)  \left(
n+2\right)  \ldots\left(  n+r+1\right)  }. \label{15}%
\end{equation}
With the help of $\left(  \ref{i4}\right)  ,$ we get%
\begin{equation}
\sum_{n=1}^{\infty}\frac{h_{n}^{\left(  r\right)  }}{\left(  n+1\right)
\left(  n+2\right)  \ldots\left(  n+r+1\right)  }=-\frac{1}{r!}\int_{0}^{1}%
\ln\left(  1-t\right)  dt. \label{16}%
\end{equation}
This equation completes the proof, since%
\begin{equation}
\int_{0}^{1}\ln\left(  1-t\right)  dt=-1. \label{17}%
\end{equation}

\end{proof}

\begin{proposition}
\label{P6}%
\begin{equation}
\sum_{n=1}^{\infty}\frac{h_{n}^{\left(  r\right)  }}{n\left(  n+1\right)
\ldots\left(  n+r\right)  }=\frac{\pi^{2}}{6r!} \label{18}%
\end{equation}

\end{proposition}

\begin{proof}
In the same way as in the proof of Proposition \ref{P5}, using $\left(
\ref{14}\right)  ,$ we can write%
\begin{equation}
\frac{1}{r!}\int_{0}^{1}h_{n}^{\left(  r\right)  }t^{n}\left(  1-t\right)
^{r}\frac{dt}{t}=\frac{h_{n}^{\left(  r\right)  }}{n\left(  n+1\right)
\ldots\left(  n+r\right)  } \label{19}%
\end{equation}
from which it follows that%
\begin{equation}
\sum_{n=1}^{\infty}\frac{h_{n}^{\left(  r\right)  }}{n\left(  n+1\right)
\ldots\left(  n+r\right)  }=-\frac{1}{r!}\int_{0}^{1}\frac{\ln\left(
1-t\right)  }{t}dt. \label{20}%
\end{equation}
Now, using the following known formula (see \cite{GR}):%
\begin{equation}
\int_{0}^{1}\frac{\ln\left(  1-t\right)  }{t}dt=-\frac{\pi^{2}}{6}, \label{21}%
\end{equation}
we obtain $\left(  \ref{18}\right)  .$
\end{proof}

\begin{remark}
Our results in Propositions \ref{P5} and \ref{P6} are involved in the Beta
function $B\left(  x,y\right)  $ defined by (see \cite{PS}):%
\begin{equation}
B\left(  x,y\right)  =\int_{0}^{1}t^{x-1}\left(  1-t\right)  ^{y-1}dt,
\label{22}%
\end{equation}
where $\operatorname{Re}\left(  x\right)  >0$ and $\operatorname{Re}\left(
y\right)  >0.$ In view of the following relation:%
\[
B\left(  r+1,n+1\right)  =\frac{\Gamma\left(  r+1\right)  \Gamma\left(
n+1\right)  }{\Gamma\left(  r+n+2\right)  }=\frac{r!n!}{\left(  r+n+1\right)
!}=\frac{r!}{\left(  n+1\right)  \left(  n+2\right)  \ldots\left(
n+r+1\right)  },
\]
the Equations $\left(  \ref{13}\right)  $ and $\left(  \ref{18}\right)
,$\ respectively, can be written in the following forms:%
\begin{equation}
\sum_{n=1}^{\infty}h_{n}^{\left(  r\right)  }B\left(  r+1,n+1\right)  =1
\label{23}%
\end{equation}
and%
\[
\sum_{n=1}^{\infty}h_{n}^{\left(  r\right)  }B\left(  r+1,n\right)  =\frac
{\pi^{2}}{6}.
\]

\end{remark}

At the end of this section we give two specific series associated with
harmonic numbers. Their proofs consist of routine manipulation with the
generating function $\left(  \ref{i4}\right)  .$ Therefore we omit proofs.

\begin{proposition}
The following equations hold.%
\[
\sum_{m=1}^{\infty}\frac{\left(  -1\right)  ^{m+1}\left(  2m+2n+3\right)
}{\left(  m+1\right)  \left(  m+2n+2\right)  }H_{m}=2\ln2\sum_{k=0}^{n}%
\frac{1}{2k+1}-\sum_{j=1}^{2n+1}\frac{1}{j}\sum_{k=1}^{j}\frac{\left(
-1\right)  ^{k-1}}{k},
\]
and%
\[
\sum_{m=1}^{\infty}\frac{\left(  -1\right)  ^{m+1}2n}{\left(  m+1\right)
\left(  m+2n+1\right)  }H_{m}=2\ln2\sum_{k=0}^{n-1}\frac{1}{2k+1}-\sum
_{j=1}^{2n}\frac{1}{j}\sum_{k=1}^{j}\frac{\left(  -1\right)  ^{k-1}}{k}.
\]

\end{proposition}


{\small \noindent\textbf{Acknowledgment}. This work is supported by the
Akdeniz University Scientific Research Projects Unit and T\"{U}B\.{I}TAK
(Scientific and Technological Research Council of Turkey).}

\end{document}